\documentclass[12pt, reqno]{amsart}
\usepackage{amsmath, amsthm, amscd, amsfonts, amssymb, graphicx, color}

\newtheorem{theorem}{Theorem}[section]

\newtheorem{corollary}[theorem]{Corollary}
\theoremstyle{definition}

\theoremstyle{remark}

\numberwithin{equation}{section}

\usepackage{bm}

\newcommand{\clb}{\mathcal{B}}

\newcommand{\cld}{\mathcal{D}}
\newcommand{\cle}{\mathcal{E}}

\newcommand{\clh}{\mathcal{H}}

\newcommand{\clo}{\mathcal{O}}

\newcommand{\D}{\mathbb{D}}

\newcommand{\raro}{\rightarrow}

\newcommand{\ke}{\clh_{\cle}(K)}

\begin{document}


\setcounter{page}{1}

\title[Linear dynamics in reproducing Kernel Hilbert spaces]{Linear dynamics in reproducing Kernel Hilbert
spaces}

\dedicatory{Dedicated to Professor Kalyan Bidhan Sinha on the
occasion of his 75th birthday}

\author[Mundayadan]{Aneesh Mundayadan}
\address{Indian Statistical Institute, Statistics and Mathematics Unit, 8th Mile, Mysore Road, Bangalore, 560059,
India}
\email{aneeshkolappa@gmail.com}

\author[Sarkar]{Jaydeb Sarkar}
\address{Indian Statistical Institute, Statistics and Mathematics Unit, 8th Mile, Mysore Road, Bangalore, 560059,
India}
\email{jay@isibang.ac.in, jaydeb@gmail.com}



\begin{abstract}
Complementing earlier results on dynamics of unilateral weighted
shifts, we obtain a sufficient (but not necessary, with supporting
examples) condition for hypercyclicity, mixing and chaos for
$M_z^*$, the adjoint of $M_z$, on vector-valued analytic reproducing
kernel Hilbert spaces $\mathcal{H}$ in terms of the derivatives of
kernel functions on the open unit disc $\mathbb{D}$ in $\mathbb{C}$.
Here $M_z$ denotes the multiplication operator by the coordinate
function $z$, that is
\[
(M_z f) (w) = w f(w),
\]
for all $f \in \mathcal{H}$ and $w \in \mathbb{D}$. We analyze the
special case of quasi-scalar reproducing kernel Hilbert spaces. We also present a complete characterization of
hypercyclicity of $M_z^*$ on tridiagonal reproducing kernel Hilbert
spaces and some special classes of vector-valued analytic
reproducing kernel Hilbert spaces.
\end{abstract}

\maketitle

\section{Introduction}

Motivated by challenges in the dynamics of bounded linear operators on
Banach spaces, in this paper, we initiate the study of dynamics of
adjoints of the multiplication operators on analytic reproducing
kernel Hilbert spaces. Here we bring together two sets of ideas -
the Hypercyclicity Criterion and analytic reproducing kernels on the
open unit disc $\D$ in $\mathbb{C}$.

More specifically, motivated by Salas' work \cite{Salas} on the
characterization of hypercyclicity of backward weighted shift
operators on $l^p$ spaces, $1 \leq p < \infty$, we propose a general
question on hypercyclicity of the adjoints of the multiplication
operators on analytic reproducing kernel Hilbert spaces. In Theorems
\ref{Main} and \ref{chaos-Mzstar}, we present sufficient conditions
for testing hypercyclicity, topological mixing and chaoticity of
adjoints of the multiplication operators on reproducing kernel
Hilbert spaces in terms of derivatives of analytic kernel functions.

In the special case of scalar-valued analytic kernel functions, our
main theorem on hypercyclicity states the following (see Theorem
\ref{Main-sub}): Let $k : \D \times \D \raro \mathbb{C}$ be an
analytic kernel function and let $\clh(k)$ denote the reproducing
kernel Hilbert space corresponding to $k$ (see Section 3 for
definitions). Suppose that the multiplication operator $M_z$ on
$\clh(k)$ is bounded. Then $M_z^*$ on $\clh(k)$ is hypercyclic if
\[
\liminf_{n}\left(\frac {1} {(n!)^2}\frac{\partial^{2n} k}{\partial
z^n\partial \bar{w}^n} (0,0) \right) = 0.
\]
In the Hilbert space setting, this result unifies previous work on
sufficient conditions for hypercyclicity of weighted shift operators
(see Gethner and Shapiro \cite{Gethner-Shapiro}, Kitai \cite{Kitai},
Rolewicz \cite{Rolewicz}, and Salas \cite{Salas}).

The sufficient condition for hypercyclicity in Theorem \ref{Main} is
not necessary, in general (see the example in Subsection
\ref{example-1}).

For the context of hypercyclicity of translation operators on
reproducing kernel Hilbert spaces of entire functions, we refer the
reader to Baranov \cite{Baranov}, Chan and Shapiro
\cite{Chan-Shapiro}, Garcia,  Hernandez-Medina and Portal
\cite{Garcia} and Godefroy and Shapiro \cite{Godefroy-Shapiro}. Also
see Bonet \cite{Bonet}.

After a preliminary section devoted to fixing the notations and
recalling the basic facts about dynamics of bounded linear
operators, in Section 3, we describe the necessary background of
reproducing kernel Hilbert spaces. Some of our results are less
standard.  Moreover, it is worth pointing out that the ``backward
shift'' behavior of the adjoints of multiplication operators on
analytic reproducing kernel Hilbert spaces in Theorem
\ref{prop-Mzstar} is of independent interest.

Section 4 contains the main results on dynamics of $M_z^*$ on
analytic reproducing kernel Hilbert spaces. In Section 5, we provide
an explicit example of an analytic reproducing kernel space on which
$M_z^*$ is hypercyclic but does not satisfy the sufficient condition
for hypercyclicity in Theorem \ref{Main}. The final section deals with the converse of Theorem
\ref{Main} for tridiagonal reproducing kernel Hilbert spaces and a class of vector valued analytic reproducing kernel Hilbert spaces.

\section{Preliminaries}\label{Sec-P}

In this section we fix some notation and recall definitions and
facts on dynamics of bounded linear operators. For more details
concerning linear dynamics, we refer the reader to the monographs by
F. Bayart and E. Matheron \cite{Bayart-Matheron} and K.G.
Grosse-Erdmann and A. Peris \cite{Erdmann-Peris}.

Let $X$ be a Banach space and let $\clb(X)$ be the set of all
bounded linear operators on $X$. Let $T$ be a bounded linear
operator on $X$. We say that $T$ is:

\noindent (i) \textit{hypercyclic} if there exists
$x\in X$ such that
\[
\{x,Tx,T^2x,\ldots\}
\]
is dense in $X$,

\noindent (ii) \textit{chaotic} if $T$ is hypercyclic and
it has a dense set of periodic points (a vector $y$ in $X$ is called
\textit{periodic} for $T$ if there exists $k \in \mathbb{N}$ such
that $T^k y = y$), and

\noindent (iii) \textit{topologically
transitive}, if
\[
\mathop{\bigcup}_{m=0}^\infty T^{-m}(U),
\]
is dense in $X$ for every non-empty open set $U$ in $X$.

From the point of view of hypercyclic operators, in this paper,
Banach spaces are always assumed to be separable.

A well known theorem of Birkhoff, known as Birkhoff's transitivity theorem,
states that (cf. \cite{Bayart-Matheron}): $T$ in $\clb(X)$ is
hypercyclic if and only if $T$ is topologically transitive.

Another important notion in dynamical systems is topological mixing:
$T$ in $\clb(X)$ is said to be \textit{topologically mixing} if,
given any two non-empty open subsets $U$ and $V$ of $X$, there
exists $N$ in $\mathbb{N}$ such that
\[
T^{-j}(U)\cap V \neq \O,
\]
for all $j\geq N$.

Classical examples of hypercyclic operators are translation
operators \cite{Birkhoff},  differential operators \cite{MacLane}
and commutators of translation operators \cite{Godefroy-Shapiro}.
The first example of a hypercyclic backward weighted shift was
produced by Rolewicz \cite{Rolewicz}: for each $1\leq p <\infty$ and
$|\alpha|>1$, the backward weighted shift $\alpha B$ is hypercyclic
on $l^p$, where $B$ on $l^p$ is the unweighted backward shift:
\[
\displaystyle B(\{c_0, c_1, c_2, \ldots\})=  \{c_1, c_2, c_3
\ldots\},
\]
for all $\{c_0, c_1, c_2 \ldots\} \in l^p$ (see also Kitai
\cite{Kitai} and Gethner and Shapiro \cite{Gethner-Shapiro}). Salas
\cite{Salas} provides a further improvement by characterizing
hypercyclicity of backward weighted shifts: Let $\lambda =
\{\lambda_n\}_{n\geq 1}$ be a bounded sequence of positive real
numbers. Then the backward weighted shift $B_{\lambda}$ on $l^p$, $1
\leq p < \infty$, defined by
\[
\displaystyle B_{\lambda} (\{c_0, c_1, c_2, \ldots\})= \{\lambda_1
c_1, \lambda_2 c_2, \lambda_3 c_3, \ldots\},
\]
for all $\{c_0, c_1, c_2, \ldots\} \in l^p$, is hypercyclic if and
only if
\[
\displaystyle \limsup_{n} \{\lambda_1 \lambda_2 \cdots
\lambda_n\}=\infty.
\]
For our purposes, we interpret the $p = 2$ case from a function
Hilbert space \cite{Shields} point of view: Given a sequence of
positive real numbers $\beta= \{\beta_n\}_{n\geq 0}$, we write $H^2(\beta)$
for the Hilbert space of all formal power series
\[
f(z)=\sum_{n \geq 0} a_n z^n,
\]
such that
\[
\displaystyle \|f\|^2_{H^2(\beta)}:=\sum_{n\geq 0} \frac {|a_n|^2}
{\beta_n^2}<\infty.
\]
Clearly, the set of functions $\{\beta_n z^n\}$ forms an orthonormal
basis in $H^2(\beta)$. If, in addition, we assume that
\[
\displaystyle \limsup_{n} \frac{\beta_{n+1}}{\beta_n} \leq 1,
\]
then $H^2(\beta) \subseteq \clo(\mathbb{D})$ (cf. Shields
\cite{Shields}), the space of all analytic functions on
$\mathbb{D}$. It also follows that the multiplication operator $M_z$
on $H^2(\beta)$ defined by
\begin{equation*}
(M_z f) (w) = w f(w),
\end{equation*}
for all $f \in H^2(\beta)$ and $w \in \D$, is bounded if and only if
\[
\sup_n \frac{\beta_n}{\beta_{n+1}} < \infty.
\]
In this case, Salas' classification result can be interpreted as
follows: $M_z^*$ is hypercyclic on $H^2(\beta)$ if and only if
\[
\liminf_n \beta_n=0.
\]
Moreover, the Costakis-Sambarino theorem \cite{Costakis-Sambarino}
states that: $M_z^*$ is mixing if and only if
\[
\lim_n \beta_n=0.
\]

The well-known and useful sufficient conditions for hypercyclicity
of operators on Banach spaces states the following, (see Kitai \cite{Kitai},
Gethner and Shapiro \cite{Gethner-Shapiro}, Godefroy and Shapiro
\cite{Godefroy-Shapiro}, and Bes and Peris \cite{Bes-Peris}): Let
$X$ be a Banach space, $D$ be a dense set in $X$, and let $T \in
\clb(X)$. If there exist a strictly increasing sequence $\{n_k\}
\subseteq \mathbb{N}$ and a map $S : D \rightarrow D$ such that

(i) $T^{n_k}S^{n_k} \rightarrow I_D$,

(ii) $T^{n_k} \rightarrow 0$, and

(iii) $S^{n_k} \rightarrow 0$,\\
pointwise in $D$, then $T$ is
hypercyclic. Moreover, if
\[
n_k=k,
\]
for all $k\in \mathbb{N}$, then $T$ is topologically mixing.

The topological mixing part in the above result is due to Costakis
and Sambarino (see Theorem 1.1, \cite{Costakis-Sambarino}).

For our purposes it is convenient to use the following version of
the Hypercyclicity Criterion (cf. Definition 1.1 in
\cite{Bernal-Erdmann} and also see \cite{Bes-Peris}):

\begin{theorem}\label{thm-hypc} \textsf{(The Hypercyclicity Criterion)}
Let $X$ be a Banach space, $D$ be a dense set in $X$, and let $T \in
\clb(X)$. Suppose that $\{n_k\} \subseteq \mathbb{N}$ is a strictly
increasing sequence. If

(i) $T^{n_k} \rightarrow 0$ pointwise on $D$, and

(ii) for each $f \in D$ there exists a sequence $\{f_k\}_{k\geq 1}
\subseteq X$ such that
\[
f_k \rightarrow 0 \quad \quad ~~~~~~~~~~~\mbox{and} \quad \quad
~~~~~~~~~T^{n_k} f_k \rightarrow f,
\]
then $T$ is hypercyclic. Moreover, if
\[
n_k=k,
\]
for all $k\in \mathbb{N}$, then $T$ is topologically mixing.
\end{theorem}

We now proceed to a chaoticity criterion (see Bonilla and
Grosse-Erdmann \cite{Bonilla-Erdmann}, page 386): Let $T$ be an
operator on a Banach space $X$, $D$ be dense in $X$ and
$S:D\rightarrow D$ be a map. If

(i) $\sum_{n\geq 0} T^n x$ and $\sum_{n\geq 0} S^n x$ are
unconditionally convergent, and

(ii) $TS x = x$,

\noindent for all $x\in D$, then $T$ is chaotic on $X$.

\noindent The above conditions are also sufficient for another
important notion in linear dynamics, called frequent hypercyclicity,
initiated by F. Bayart and S. Grivaux \cite{Bayart-Grivaux}. We
refer the reader to \cite{Bayart-Matheron} and
\cite{Bonilla-Erdmann} for an introductory discussion of the topic.

Recall that a series $\sum_{n} u_n$ in a Banach space $X$ is said to
be \textit{unconditionally~convergent} if $\sum_n u_{\sigma(n)}$ is
convergent for all permutations $\sigma$ on $\mathbb{N}$. It is well
known that the unconditional convergence of $\sum_n u_n$ is
equivalent to the following: for $\epsilon>0$, there exists $N\in
\mathbb{N}$ such that
\begin{equation*}
\Big \|\sum_{n\in F}u_n\Big \|<\epsilon,
\end{equation*}
for all finite subsets $F$ of $\{N,N+1,N+2,\cdots\}$ (see, for
instance \cite{Bayart-Matheron}, page 138).

The following version of the Chaoticity Criterion will be useful in what
follows. This is probably known to experts, though we cannot find
an exact reference (however, see Remark 9.10 in \cite{Erdmann-Peris}
and page 386 in \cite{Bonilla-Erdmann}).

\begin{theorem}\textsf{(Chaoticity Criterion)} \label{chaos}
Let $X$ be a Banach space,  $T \in \clb(X)$ and $D$ be a dense set.
Then $T$ is chaotic if for each $x\in D$, there exists a sequence
$\{u_k\}_{k\geq 0}$ in $X$ with $u_0=x$ such that
\[
\displaystyle \sum_{n\geq 0} T^n x \quad \quad \mbox{and} \quad
\quad \sum_{n\geq 0} u_n,
\]
are unconditionally convergent, and
\[
T^nu_k=u_{k-n},
\]
for all $k\geq n$.
\end{theorem}

\begin{proof}
Observe that if $x \in D$, then
\[
T^n x \rightarrow 0 ~~\text{and}~~f_n:=u_n\rightarrow 0,
\]
as $n \raro \infty$. Moreover,
\[
T^nf_n=T^nu_n=u_{n-n}=u_0=x,
\]
for all $n \geq 0$, and hence, in particular
\[
T^nf_n \raro x,
\]
as $n \raro \infty$. This shows that $T$ satisfies the
Hypercyclicity Criterion with respect to the same dense set $D$ and
the sequence $n_k =k$ for all $k\in \mathbb{N}$, and hence $T$ is
hypercyclic. It remains to show that $T$ has a dense set of periodic
points in $X$. Define
\begin{equation*}
x_p = \sum_{n\geq 1} T^{np} x + x + \sum_{n\geq 1} u_{np},
\end{equation*}
for all $p\geq 1$. By the unconditional convergence of the given
series we have that $x_p\rightarrow x$, as $p\rightarrow \infty$. On
the other hand
\begin{equation*}
T^p x_p = (\sum_{n\geq 1} T^{(n+1)p} x) + T^p x + (\sum_{n\geq 1}
T^p u_{np})
\end{equation*}
together with the properties of $\{u_n\}$ gives
\[
T^p(x_p)=x_p,
\]
that is, $x_p$ is a periodic point for $T$. The result now follows
from the fact that $D$ is dense in $X$ and $\{x_p\}$ approximates
the element $x$ in $D$.
\end{proof}

For our purposes (cf. Theorem \ref{chaos-Mzstar}) it is also
relevant to recall the complete characterization of chaoticity for
backward weighted shifts $B_{\lambda}$ on $l^p$ spaces (see
Grosse-Erdmann \cite{Erdmann}): $B_\lambda$ on $l^p$ is chaotic if
and only if
\[
\sum_n(\lambda_1 \lambda_2 \cdots \lambda_n)^{-p} < \infty.
\]

\noindent The case $p=2$ yields that $M_z^*$ on $H^2(\beta)$ is
chaotic if and only if
\[
\sum_n\beta_n^{2} < \infty.
\]

\section{Reproducing kernel Hilbert spaces and derivatives}

We first recall the basics and constructions of reproducing kernel
Hilbert spaces (see \cite{Aronszajn} and \cite{Paulsen}). Let $\cle$
be a Hilbert space. An operator-valued function $K : \D \times \D
\raro \clb(\cle)$ is called an \textit{analytic kernel} if $K$ is
analytic in the first variable and
\[
\sum_{i,j = 1}^n \langle K(z_i, z_j) \eta_j, \eta_i \rangle_{\cle}
\geq 0,
\]
for all $\{z_i\}_{i=1}^n \subseteq \D$ and $\{\eta_i\}_{i=1}^n
\subseteq \cle$ and $n \in \mathbb{N}$. In this case there exists a
Hilbert space $\clh_{\cle}(K)$ of $\cle$-valued analytic functions
on $\D$ such that
\[
\{K(\cdot, w) \eta : w \in \D, \eta \in \cle\},
\]
is a total set in $\clh_{\cle}(K)$. Here for $w\in \mathbb{D}$ and
$\eta\in \cle$, the symbol $K(\cdot, w)\eta$ represents the function
\[
(K(\cdot, w)\eta) (z) = K(z,w)\eta,
\]
for all $z\in \mathbb{D}$. It is easy to verify that
\[
\langle f, K(\cdot, w) \eta \rangle_{\clh_{\cle}(K)} = \langle f(w),
\eta \rangle_{\cle},
\]
for all $f\in \ke$, $w \in \D$ and $\eta \in \cle$ (cf.
\cite{Kumari-etal}). In other words
\begin{equation}\label{eq-ev}
\langle ev_w f, \eta \rangle_{\cle} = \langle f, K(\cdot, w) \eta
\rangle_{\clh_{\cle}(K)},
\end{equation}
where $ev_w : \clh_{\cle}(K) \raro \cle$ is the (bounded) evaluation
map defined by
\[
ev_w f = f(w),
\]
for all $w \in \D$ and $f \in \clh_{\cle}(K)$. We call the Hilbert
space $\ke$ the \textit{analytic reproducing kernel Hilbert space}
corresponding to the kernel $K$.

\noindent It also follows from \eqref{eq-ev} that the functions
\[
w \mapsto ev_w \in \clb(\ke, \cle),
\]
and
\[
w \mapsto K(z, w) \in \clb(\cle) \quad \quad (z\in \mathbb{D}),
\]
are analytic and co-analytic on $\D$, respectively, and
\[
K(z, w)^* = K(w, z),
\]
and
\[
ev_z \circ ev_w^* = K(z, w),
\]
for all $z, w \in \D$. Furthermore, note that
\[
\begin{split}
\|K(\cdot, w) \eta\|^2_{\ke} & = \langle K(\cdot, w) \eta, K(\cdot,
w) \eta \rangle_{\clh_{\cle}(K)}
\\
& = \langle K(w, w)\eta, \eta\rangle_\cle,
\end{split}
\]
that is
\begin{equation}\label{eq-Knorm}
\|K(\cdot, w) \eta\|_{\ke} =\| K(w, w)^{1/2}\eta\|_\cle,
\end{equation}
for all $w \in \D$ and $\eta \in \cle$.

Conversely, let $\clh$ be a Hilbert space of analytic functions on
$\mathbb{D}$ taking values in $\cle$, and let the evaluation map
$ev_w : \clh \raro \cle$ (defined as above) be continuous for all $w
\in \D$. Then $\clh$ is an analytic reproducing kernel Hilbert space
corresponding to the $\clb(\cle)$-valued analytic kernel $K$ on $\D$
where $K(z, w) = ev_z \circ ev_w^*$ for all $z, w \in \D$.

Now let $\ke$ be an analytic reproducing kernel Hilbert space. Note
that the $\clb(\clh_\cle(K), \cle)$-valued function $z\mapsto
ev_{z}$ is analytic on $\mathbb{D}$. Indeed, for fixed $f\in
\clh_\cle(K)$ and $\eta \in \cle$, the function
\[
z\mapsto \langle ev_z(f),\eta\rangle_{\cle} = \langle
f(z),\eta \rangle_{\cle},
\]
is analytic on $\D$. Since analyticity in the strong operator
topology is equivalent to the analyticity in operator norm (cf.
Chapter 5, Theorem 1.2, \cite{Taylor-Lay}), it follows that
$z\mapsto ev_{z}$ is analytic on $\D$. Now if $z_0\in \mathbb{D}$,
then we have the following power series (in powers of $(z-z_0)$)
expansion
\[
ev_{z}= \sum_{n \geq 0} (z-z_0)^n \frac{\partial^n ev_z}{\partial
z^n}\Big |_{z=z_0},
\]
in the operator norm topology. This implies that $\frac{\partial^n
ev_z}{\partial z^n}\Big |_{z=z_0} :\clh_\cle(K) \raro \cle$ is a
bounded linear operator for each $z_0\in \D$ and $n \geq 0$.
Moreover, since
\[
\begin{split}
\Big(\frac{ev_{z_0 + h} - ev_{z_0}}{h}\Big) f & = \frac{f(z_0 + h) -
f(z_0)}{h}
\\
& \raro f'(z_0),
\end{split}
\]
as $h \raro 0$, we have that
\[
\begin{split}
\frac{\partial  ev_z}{\partial z}\Big |_{z=z_0}(f)= f'(z_0),
\end{split}
\]
for all $f \in \ke$ and $z_0\in \D$. Similarly
\[
\begin{split}
\frac{\partial^n ev_z}{\partial z^n}\Big |_{z=z_0} (f) = f^{(n)}(z_0),
\end{split}
\]
for all $f \in \ke$ and $z_0 \in \D$. On the other hand, the same
argument as above can now be used to show that $\frac{\partial^n
ev_w^*}{{\partial} \bar{w}^n}\Big |_{w=w_0} : \cle \raro \ke$ is a
bounded linear operator for each $w_0 \in \D$ and $n \geq 0$. Now
using the identity
\[
ev_{\lambda}^* \eta = K(\cdot, \lambda) \eta,
\]
for all $\lambda \in \D$ and $\eta \in \cle$, we find, for each $w_0
\in \D$, that
\[
\begin{split}
\Big(\frac{ev^*_{w_0 + h} - ev^*_{w_0}}{\overline{h}}\Big) \eta & =
\frac{K(\cdot, w_0 + h) \eta - K(\cdot, w_0) \eta}{\overline{h}}
\\
& \raro \Big (\frac{\partial K}{\partial \bar{w}} (\cdot, w_0)\Big)
\eta,
\end{split}
\]
as $h \raro 0$. Thus
\[
\frac{\partial ev_w^*}{{\partial} \bar{w}}\Big|_{w=w_0}(\eta) =  \Big(\frac{\partial
K}{\partial \bar{w}} (\cdot, w_0)\Big) \eta,
\]
and, again, we have that
\[
\frac{\partial^n ev_w^*}{{\partial} \bar{w}^n}\Big|_{w=w_0}(\eta) =
\Big(\frac{\partial^n K}{\partial \bar{w}^n} (\cdot, w_0)\Big) \eta,
\]
for all $w_0 \in \D$, $\eta \in \cle$ and $n \geq 0$. Moreover, we
have
\[
\begin{split}
\frac{\partial K(\cdot,w)}{\partial
\bar{w}}\Big|_{w=w_0}(\eta)&=\lim_{h\rightarrow 0}
\frac{ev_{w_0+h}^*(\eta)-ev_{w_0}^*(\eta)}{\overline{h}}\\&=\Big
(\frac{\partial ev_z}{\partial z}\Big |_{z=w_0}\Big )^*\eta,
\end{split}
\]
and similarly,
\begin{equation}\label{adjoint}
\frac{\partial^n K(\cdot,w)}{\partial
\bar{w}^n}\Big|_{w=w_0}=\frac{\partial^nev_z}{\partial z^n}\Big
|_{z=w_0}^*,
\end{equation}
for all $n\in \mathbb{Z}_+$. Furthermore, for all $f \in \ke$, $n
\in \mathbb{Z}_+$, $w \in \D$ and $\eta \in \cle$, it follows that
\[
\begin{split}
\Big \langle f, \Big(\frac{\partial^n K}{\partial \bar{w}^n} (\cdot, w)
\Big|_{w=w_0}\Big)\eta \Big \rangle_{\ke} & = \frac{\partial^n}{\partial {w}^n}
\Big(\langle f,K(\cdot,w)\eta\rangle_{\ke} \Big)\Big |_{w=w_0}
\\
& = \frac{\partial^n}{\partial
{w}^n} \Big( \langle f(w),\eta\rangle_\cle \Big)\Big |_{w=w_0} \\
& = \langle f^{(n)}(w_0), \eta \rangle_\cle,
\end{split}
\]
that is
\[
\langle f^{(n)}(w_0), \eta \rangle_\cle = \Big\langle f, \frac{\partial^n
ev_w^*}{{\partial} \bar{w}^n}\Big |_{w=w_0} \eta \Big\rangle_{\ke}.
\]
In particular, if $n \in \mathbb{Z}_+$ and $\eta \in \cle$,
then the function
\[
z\mapsto \Big(\frac{\partial^n K}{\partial \bar{w}^n}
 (z, 0)\Big) \eta,
\]
is in $\ke$, which follows from \eqref{adjoint}. Therefore, we have
proved the following:

\begin{theorem}\label{collect}
Let $\ke$ be an analytic reproducing kernel Hilbert space. Let
\[
K_{n} (z) = \frac{\partial^n K}{\partial \bar{w}^n} (z, 0),
\]
and let
\[
K_{n, \eta} (z) = K_n(z) \eta,
\]
for all $z \in \D$, $\eta \in \cle$ and $n \in \mathbb{Z}_+$. Then
$K_{n} (z) \in\clb(\cle)$ for all $z \in \D$, $K_{n, \eta} \in \ke$
and
\[
\langle f^{(n)}(0), \eta \rangle_{\cle} = \langle f, K_{n, \eta}
\rangle_{\ke},
\]
for all $f \in \ke$, $\eta \in \cle$ and $n \in \mathbb{Z}_+$.
\end{theorem}

The following formulae for the inner product and norm are now
immediate:

\begin{corollary}\label{cor-normK}
In the setting of Theorem \ref{collect}, we have
\[
\langle K_{m,\zeta},K_{n,\eta} \rangle_{\ke}= \Big\langle
\Big(\frac{\partial^{n+m} K}{\partial z^n\partial \bar{w}^m} (0, 0)
\Big) \zeta, \eta \Big\rangle_{\cle},
\]
for all $m, n \in \mathbb{Z}_+$ and $\eta, \zeta \in \cle$. In
particular, we have
\[
\|K_{n,\eta} \|_{\ke}^2 = \Big\langle \Big(\frac{\partial^{2n}
K}{\partial z^n\partial \bar{w}^n} (0, 0) \Big) \eta, \eta
\Big\rangle_{\cle}.
\]
\end{corollary}

\begin{proof}
By using the identity in Theorem \ref{collect} with $f=K_{m,\zeta}$,
the result follows.
\end{proof}

For example, in the particular case of a scalar-valued kernel $K$, we
have
\[
\displaystyle K(z,w)=\sum_{m, n \geq 0} a_{mn} z^m \bar{w}^n,
\]
for some (not necessarily bounded) infinite matrix $(a_{mn})$.
Then the inner product and norm expressions in Corollary
\ref{cor-normK} are given by $n!m!a_{m n}$ and $(n!)^2a_{n n}$,
respectively.

Theorem \ref{collect} and Corollary \ref{cor-normK} should be
compared with Lemmas 4.1 and 4.3 in \cite{CS} on generalized Bergman
kernels.

Now let $M_z$ denote the multiplication operator on $\ke$, that
is
\[
(M_z f)(w) = w f(w),
\]
for all $f \in \ke$ and $w \in \D$. If $M_z$ is bounded, $w \in \D$,
$\eta \in \cle$ and $f \in \ke$, then
\[
\begin{split}
\langle M_z^*(K(\cdot,w)\eta), f\rangle_{\ke} & = \langle K(\cdot,
w)\eta, z f \rangle_{\ke}
\\
& = \langle\eta, w f(w) \rangle_{\cle}
\\
& = \langle \bar{w}\eta,f(w)\rangle_{\cle}
\\
& = \langle\bar{w}K(\cdot, w) \eta, f \rangle_{\ke},
\end{split}
\]
that is
\begin{equation}\label{eigenvalue}
M_z^* \big(K(\cdot, w) \eta\big) = \bar{w} K(\cdot, w) \eta.
\end{equation}
In particular
\begin{equation}\label{eq-k0}
M_z^*(K(\cdot,0)\eta)=0,
\end{equation}
for all $\eta \in \cle$. Now using
\[
K(\cdot, w) \eta = \sum_{n=0}^\infty \frac{1}{n!} K_{n, \eta}
\bar{w}^n,
\]
it follows from \eqref{eigenvalue} that
\[
M_z^* (\sum_{n=0}^\infty \frac{1}{n!} K_{n, \eta} \bar{w}^n) =
\bar{w} \sum_{n=0}^\infty \frac{1}{n!} K_{n, \eta} \bar{w}^n,
\]
for all $w \in \D$ and $\eta \in \cle$. Equating the coefficients of
the same powers of $\bar{w}$ on either side, we immediately get the
following ``backward shift'' property of $M_z^*$:

\begin{theorem}\label{prop-Mzstar}
Let $\ke$ be an analytic reproducing kernel Hilbert space. If the
multiplication operator $M_z$ is bounded on $\ke$, then
\[
M_z^*(\frac{1}{n!} K_{n, \eta}) = \begin{cases} \frac{1}{(n-1)!}
K_{n-1, \eta} & \mbox{if}~ n \geq 1
\\
0 & \mbox{if}~n=0, \end{cases}
\]
for all $\eta \in \cle$.
\end{theorem}

\section{Dynamics of $M_z^*$}

In this section we present the main results on dynamics of $M_z^*$
on $\clh_{\cle}(K)$. Our first result concerns hypercyclicity and
topological mixing. The second result is about the chaoticity of
$M_z^*$.

We begin with the following definition: Let $I$ be a set, and let
$\{u_n\}$ be a sequence of complex-valued functions on $I$. We say
that $\displaystyle \liminf_n u_n(x)=0$ uniformly in $x\in I$, if
there exists a subsequence $\{n_k\}$ of natural numbers such that
$\displaystyle \lim_{k \raro \infty} u_{n_k}(x) = 0$ uniformly in
$x\in I$.

\begin{theorem}\label{Main}
Let $\cle$ be a Hilbert space, $\cle_0$ a total subset in $\cle$ and
let $\ke$ be an analytic reproducing kernel Hilbert space. If $M_z$
is bounded on $\ke$, then the following hold:

(1) If
\[
\liminf_{n}\left(\frac {1} {(n!)^2}\Big< \Big(
\frac{\partial^{2n} K}{\partial z^n\partial \bar{w}^n} (0,0) \Big)
\eta,\eta \Big>_{\cle}\right)=0,
\]
uniformly in $\eta\in \cle_0$, then $M_z^*$ is hypercyclic.

(2) If
\[
\lim_{n}\left(\frac {1} {(n!)^2}\Big< \Big(
\frac{\partial^{2n} K}{\partial z^n\partial \bar{w}^n} (0,0) \Big)
\eta,\eta \Big>_{\cle}\right)=0,
\]
for all $\eta\in \cle_0$, then $M_z^*$ is topologically mixing.
\end{theorem}

\begin{proof}
We apply Theorem \ref{thm-hypc}, the Hypercyclicity Criterion, to
$M_z^*$. For each $\eta \in \cle_0$ and $m \in \mathbb{Z}_+$, define
$\hat{K}_{m,\eta} \in \clh_{\cle}(K)$ (see Theorem \ref{collect}) by
\[
\hat{K}_{m, \eta} (z) = \frac{1} {m!}\Big(\frac{\partial^m
K}{\partial \bar{w}^m} (z, 0) \Big) \eta.
\]
Then the set $D$ is total in $\ke$, where
\begin{equation*}
\displaystyle D = \mbox{span} \{\hat{K}_{m,\eta}:~m \in
\mathbb{Z}_+, \eta \in \cle_0\}.
\end{equation*}
Indeed, if $g \in \ke$ is orthogonal to $D$, then
\[
\langle g,\hat{K}_{m,\eta}\rangle=0,
\]
and so, by Theorem \ref{collect}, we have
\begin{center}
$\displaystyle \langle g^{(m)}(0),\eta\rangle_\cle=0$,
\end{center}
for all $m \in \mathbb{Z}_+$ and $\eta \in \cle_0$. Since the span of $\cle_0$ is dense in $\cle$, it follows that
\[
g^{(m)}(0)=0,
\]
for all $m \in \mathbb{Z}_+$ and hence
\[
g\equiv 0.
\]
Now note that $\{\hat{K}_{m,\eta}:~m \in \mathbb{Z}_+, \eta \in
\cle_0\}$ is a generating set for $D$. If $\zeta \in \cle$,  $s > t$
and $s, t \in \mathbb{Z}_+$, then by \eqref{eq-k0} and Theorem
\ref{prop-Mzstar}, it follows that
\begin{equation}\label{eq-starzero}
M_z^{* s} \hat{K}_{t, \zeta} = 0.
\end{equation}
Let
\begin{equation*}
f = \sum_{j=1}^p\alpha_j \hat{K}_{m_j,\eta_j}\in D,
\end{equation*}
for some positive integer $p$, $\alpha_j\in \mathbb{C}$, $m_j\in
\mathbb{Z}_+$ and $\eta_j\in \cle_0$, $1\leq j\leq p$. Then
\eqref{eq-starzero} implies, in particular, that
\[
M_z^{*n} f \raro 0,
\]
as $n \raro \infty$. This proves the first condition of the
criterion.

\noindent To prove the second condition of the criterion, let
\[
f=\sum_{j=1}^p\alpha_j \hat{K}_{m_j,\eta_j}\in D,
\]
and first define a sequence $\{g_n\} \subseteq D$ by
\begin{equation*}
g_n:=\sum_{j=1}^p\alpha_j \hat{K}_{m_j+n,\eta_j},
\end{equation*}
for all $n \in \mathbb{N}$. By Theorem \ref{prop-Mzstar}, it follows
that
\[
M_z^{*n} (\hat{K}_{m_j+n, \eta_j}) = \hat{K}_{m_j,\eta_j},
\]
for $j = 1, \ldots, p$, and hence
\begin{equation}\label{eq-prf1}
M_z^{*n} g_n= f,
\end{equation}
for $n \in \mathbb{N}$. Since
\[
\|\hat{K}_{m, \eta}\|_{\ke}^2 = \frac {1} {(m!)^2}\Big< \Big(
\frac{\partial^{2m} K}{\partial z^m\partial \bar{w}^m} (0,0) \Big)
\eta,\eta \Big>_{\cle},
\]
for all $m \geq 0$, by Corollary \ref{cor-normK}, and
\[
\liminf_{n}\left(\frac {1} {(n!)^2}\Big< \Big( \frac{\partial^{2n}
K}{\partial z^n\partial \bar{w}^n} (0,0) \Big) \eta,\eta
\Big>_{\cle}\right)=0,
\]
uniformly in $\eta\in \cle_0$ by assumption, it follows that there
exists a sequence of natural numbers $\{t_k\}$ such that
\[
\hat{K}_{{t_k},\eta}\rightarrow 0,
\]
in $\ke$, for all $\eta \in \cle_0$. This implies that
\[
M_z^{*j} (\hat{K}_{{t_k},\eta}) \raro 0,
\]
as $k \raro \infty$ and for all $j \geq 0$. Now observe that
\[
\hat{K}_{{t_k-j},\eta} = M_z^{*j} (\hat{K}_{{t_k},\eta}),
\]
for all $t_k \geq j$. Hence
\[
\hat{K}_{{t_k-j},\eta} \rightarrow 0,
\]
as $k\rightarrow \infty$ and for all $j \geq 0$.

\noindent We now use the following lemma (see Lemma 4.2, page 90,
\cite{Erdmann-Peris}): Let $(X, d)$ be a metric space and let
$\{y_k\}$ be a sequence in $X$. Suppose that $\{t_k\}$ is a sequence
of natural numbers. If the subsequence $\{y_{t_k-j}\}$ converges to
a fixed element $y$ for each $j$, then there exists $\{n_k\}$ such
that $\{y_{n_k+j}\}$ converges to $y$ for each $j$.

\noindent This shows that there exists a strictly increasing
sequence of natural numbers $\{n_k\}_k$ such that
\[
 \hat{K}_{n_k+j,\eta} \rightarrow 0,
\]
for each $j$ and for all $\eta \in \cle_0$. Set
\[
f_k: = g_{n_k},
\]
that is
\[
f_k= \sum_{j=1}^{p}\alpha_j \hat{K}_{n_k + m_j,\eta_j},
\]
for all $k \geq 1$. Clearly
\[
f_{k} \raro 0,
\]
as $k \raro \infty$. Moreover, by \eqref{eq-prf1} we have
\[
M_z^{*n_k} f_{k}= f,
\]
for all $k \geq 1$. Hence $M_z^*$ satisfies the Hypercyclicity
Criterion with respect to $\{n_k\}$. This proves (1).

\noindent For the second part, we proceed with $D$ and $f$ as in the
proof of part (1) above. In this case, however, by assumption, it
follows that
\[
\hat{K}_{n, \eta} \raro 0,
\]
as $n \raro \infty$ and for all $\eta \in \cle_0$. If we set
\[
f_n = \sum_{j=1}^{p}\alpha_j \hat{K}_{n + m_j,\eta_j},
\]
for all $n \geq 1$, then, as in the proof of part (1), it follows
that
\[
f_{n} \raro 0,
\]
as $n \raro \infty$, and
\[
M_z^{*n} f_{n}= f,
\]
for all $n \geq 1$. This concludes the proof.
\end{proof}

We now proceed to the chaos of $M_z^*$ on $\ke$. A double series
$\displaystyle \sum_{i,j\geq 0} u_{i,j}$ in a Hilbert space $\clh$
is called $\mathcal{F}$-summable if for $\epsilon>0$ there exists
$N\in \mathbb{N}$ such that
\[
\|\sum_{i,j\in F} u_{i,j}\|<\epsilon,
\]
for all finite sets $F\subset \{N,N+1,\cdots\}$.

\begin{theorem}\label{chaos-Mzstar}
Let $\cle$ be a Hilbert space, $\cle_0$ a total subset in $\cle$ and
let $\ke$ be an analytic reproducing kernel Hilbert space. If $M_z$
is bounded on $\ke$ and the double series
\[
\sum_{n,m\geq 0}  \frac {1} {n!~m!}\Big\langle \Big(\frac{\partial^{n+m}
K}{\partial z^n\partial \bar{w}^m} (0,0) \Big) \eta,\eta
\Big\rangle_{\cle},
\]
is $\mathcal{F}$-summable for all $\eta \in \cle_0$, then $M_z^*$
chaotic.
\end{theorem}

\begin{proof}
Here we apply Theorem \ref{chaos}, the Chaoticity Criterion, and
proceed with the set $D$ as defined in the proof of Theorem
\ref{Main}. Assume that
\begin{equation*}
f:=\sum_{j=1}^p\alpha_j \hat{K}_{m_j,\eta_j}\in D,
\end{equation*}
and choose
\[
 u_k:=\sum_{j=1}^{p}\alpha_j \hat{K}_{k+m_j,\eta_j} \quad \quad (k\geq
 0).
\]
From the proof of Theorem \ref{Main} it follows that for $k\geq n$,
\[
M_z^{*n} u_k =\sum_{j=1}^{p}\alpha_j \hat{K}_{k-n+m_j,\eta_j}=u_{k-n}.
\]
\noindent Now using \eqref{eq-starzero} we obtain that the series
\[
\sum_{n\geq 0} M_z^{*n} \hat{K}_{m,\eta},
\]
is unconditionally convergent in $\clh_{\cle}(K)$ for each $m \in
\mathbb{Z}_+$ and $\eta \in \cle_0$, and so
\[
\sum_{n\geq 0} M_z^{*n} f,
\]
is also unconditionally convergent in $\clh_{\cle}(K)$ for all $f
\in D$. Now, we verify that the series $\displaystyle \sum_{n\geq 0}
u_n$ is unconditionally convergent. It is, however, enough to check
the above for
\[
f = \hat{K}_{m, \eta},
\]
where $m \in \mathbb{Z}_+$ and $\eta \in \cle_0$, that is, we need
to verify that the series
\[
\sum_{n \geq 0}\hat{K}_{m + n,\eta}
\]
is unconditionally convergent in $\ke$. This is equivalent to
showing that
\[
\sum_{n \geq 0}\hat{K}_{ n,\eta},
\]
is unconditionally convergent for all $\eta \in \cle_0$. Indeed, for
$\epsilon>0$ and $\eta \in \cle_0$, by hypothesis, there exists $N
\in \mathbb{N}$ such that
\begin{equation*}
\sum_{n,m \in F} \frac {1} {n!~m!} \Big\langle \Big(
\frac{\partial^{n+m} K}{\partial z^n\partial \bar{w}^m} (0,0) \Big)
\eta,\eta \Big\rangle_{\cle} <\epsilon,
\end{equation*}
for all finite sets $F\subset [N,\infty)\cap \mathbb{N}$. Then
Corollary \ref{cor-normK} implies that
\begin{equation*}
\Big \langle \sum_{n\in F}\hat{K}_{n,\eta},\sum_{m\in F} \hat{K}_{m,\eta}
\Big \rangle<\epsilon,
\end{equation*}
that is
\[
\left\|\sum_{n\in F} \hat{K}_{n,\eta}\right\|_{\ke}<\sqrt{\epsilon}.
\]
Hence the series $\sum_{n\geq 0} \hat{K}_{n,\eta}$ is
unconditionally convergent in $\ke$ (see the discussion preceding
Theorem \ref{chaos}). This completes the proof of the theorem.
\end{proof}

Of particular interest is the case where $K$ is a scalar-valued
kernel. Here the sufficient condition for dynamics of $M_z^*$
involves the diagonal of the kernel function. We will address this
issue in Theorem \ref{Main-sub}.

The converse of Theorem \ref{Main} is not true in general (see the
example in Subsection \ref{example-1} and also Section
\ref{sect-CR}).

\section{Examples and Applications}

In this section we give concrete examples and applications of our
main results.

\subsection{Weighted shifts}\label{example-1}

We begin by noting that the weighted shift space $H^2(\beta)$, as
introduced in Section 2, is an analytic reproducing kernel Hilbert
space. In this case, the scalar-valued reproducing kernel function
is given by
\[
K(z,w)=\sum_{n\geq 0} \beta_n^2 z^n\overline{w}^n,
\]
for all $z, w \in \D$. Moreover
\[
\frac {\partial^{n+m}K}{\partial z^n\bar{w}^m} (0,0) =
\begin{cases} (n!)^2
\beta_n^2 & \mbox{if}~ m = n
\\
0 & \mbox{if}~m \neq n, \end{cases}
\]
and $m, n \geq 0$. By Theorem \ref{Main}, it follows that $M_z^*$ is
hypercyclic on $H^2(\beta)$ if
\[
\liminf_n \beta_n = 0.
\]
This condition is also necessary for hypercyclicity  of $M_z^*$ on
$H^2(\beta)$ (see Salas \cite{Salas}) but not in general on
$\clh_{\cle}(K)$ (see the Subsection \ref{example-1}). A similar
classification result also holds for chaoticity and mixing for
$M^*_z$ on $H^2(\beta)$ (cf. \cite{Erdmann-Peris}).

\subsection{Quasi-scalar reproducing kernel Hilbert
spaces}\label{sub-qs}

We say that a kernel function $K : \D \times \D
\raro \clb(\cle)$ is a quasi-scalar kernel if there exists a
scalar-valued analytic kernel $k : \D \times \D \raro \mathbb{C}$
such that
\[
K(z,w) = k(z,w) I_{\cle},
\]
for all $z, w \in \D$. In this case, the general construction of
reproducing kernel Hilbert spaces yields that (for instance, see
\cite{Kumari-etal})
\[
\clh_{\cle}(K) \cong \clh(k) \otimes \cle,
\]
where $\clh(k)$ is the reproducing kernel Hilbert space
corresponding to the kernel function $k$ on $\D$. Quasi-scalar
reproducing kernel Hilbert spaces play important roles in function
theory and operator theory, particularly in the study of dilation
theory and analytic model theory (cf. \cite{Kumari-etal}).

As a simple example of the use of Theorems \ref{Main} and
\ref{chaos-Mzstar}, we prove the following:

\begin{theorem} \label{Main-sub}
If $\clh_K(\cle)$ is a quasi-scalar reproducing kernel Hilbert space
and $M_z$ on $\clh_K(\cle)$ is bounded, then:

\noindent (1) $M_z^*$ is hypercyclic  if
\[
\displaystyle \liminf_{n} \frac {1} {(n!)^2}\frac{\partial^{2n}
k}{\partial z^n\partial \bar{w}^n}(0,0)=0.
\]
(2) $M_z^*$ is topologically mixing if
\[
\displaystyle \lim_n \frac {1} {(n!)^2}\frac{\partial^{2n}
k}{\partial z^n\partial \bar{w}^n}(0,0)=0.
\]
(3) $M_z^*$ is chaotic if the series
\[
\sum_{n,m\geq 0}  \frac {1} {n!~m!}\frac{\partial^{n+m}
k}{\partial z^n\partial \bar{w}^m}(0,0),
\]
is $\mathcal{F}$-summable.
\end{theorem}
\begin{proof}
If $k : \D \times \D \raro \mathbb{C}$ is an analytic kernel on $\D$
and
\[
K(z,w) = k(z,w) I_{\cle},
\]
for all $z, w \in \D$, then
\[
\|K(\cdot, w)\eta\|_{\ke}=\|k(\cdot, w)\|_{\clh(k)} \|\eta\|_\cle,
\]
for all $\eta \in \cle$. Observe now that
\[
\Big\langle \Big( \frac{\partial^{{n+m}} K}{\partial z^n \partial
\bar{w}^m} (0,0) \Big) \eta,\eta \Big\rangle_{\cle} = \Big(
\frac{\partial^{n+m} k}{\partial z^n\partial \bar{w}^m} (0,0) \Big)
\|\eta\|_{\cle}^2,
\]
for all $w \in \D$, $\eta \in \cle$ and $n,m \geq 0$. The result now
follows from Theorems \ref{Main} and \ref{chaos-Mzstar}.
\end{proof}

It is worth pointing out that the conclusion of the above theorem is
independent of the choice of the Hilbert space $\cle$. More
specifically, if $k$ is a scalar kernel and
\[
\displaystyle \liminf_{n} \frac {1} {(n!)^2}\frac{\partial^{2n}
k}{\partial z^n\partial \bar{w}^n}(0,0)=0,
\]
then $M_z^*$ on $\clh_{\cle}(K)$ is hypercyclic where
\[
K(z, w) = k(z, w) I_{\cle},
\]
for all $z, w \in \D$ and $\cle$ is a Hilbert space. This
observation also should be compared with the linear dynamics of
tensor products of operators (cf. \cite{Bonet-etal} and
\cite{Gimenez-Peris}).

As a concrete application, consider the $\cle$-valued Dirichlet
space $\mathcal{D}_\cle$ on $\mathbb{D}$, where $\cle$ is separable
Hilbert space. Notice that the kernel function for $\cld_{\cle}$ is
given by
\[
(z, w) \mapsto \Big(\sum_{n\geq 0} \frac{z^n\bar{w}^n}{n+1} \Big)
I_\cle.
\]
From the above theorem, it then follows that $M_z^*$ is mixing on
$\mathcal{D}_\cle$.

\subsection{A counter-example}\label{example-1}

Here we present a counterexample to show that the sufficient
condition in Theorem \ref{Main-sub} for $M_z^*$ to be hypercyclic is
not a necessary condition.

\noindent Consider the Hilbert space $H^2(\beta) \subseteq
\clo(\D)$, as in Section \ref{Sec-P}, corresponding to the
(diagonal) kernel
\[
k(z, w) = \sum_{n \geq 0}\beta_n^2 z^n \bar{w}^n \quad \quad (z, w
\in \D),
\]
where $\{\beta_n\}$ is a sequence of positive real numbers and
\[
\displaystyle \limsup_{n} \frac{\beta_{n+1}}{\beta_n} \leq 1.
\]
Suppose that $M_{z, k}$, the multiplication operator by the
coordinate function $z$, on $H^2(\beta)$ is bounded. Let
\[
\theta(z)=\frac {1} {1-z},
\]
for all $z \in \D$, and set
\[
\mathcal{H} = \{\theta f : f\in H^2(\beta)\}.
\]
Then $\clh$ is a Hilbert space with the inner product
\[
\langle \theta f,\theta g \rangle_\mathcal{H}: = \langle
f,g\rangle_{H^2(\beta)},
\]
for all $f, g \in H^2(\beta)$. Moreover, $\clh$ is an analytic
reproducing kernel Hilbert space corresponding to the kernel
\[
k_{\theta}(z, w) = \theta(z)\left(\sum_{n\geq 0} {\beta_n}^2 z^n\bar{w}^n \right)\overline{\theta(w)},
\]
for all $z, w \in \D$. Here the reproducing property is given by
\[
\langle \theta f, k_{\theta}(\cdot,w)\rangle_\mathcal{H} = \langle
\theta f,\overline{\theta(w)} \theta k(\cdot,w)
\rangle_\mathcal{H}=\theta(w) \langle f,k(\cdot,w)
\rangle_{H^2(\beta)} =\theta(w)f(w),
\]
for all $f \in H^2(\beta)$ and $w \in \D$. Since
\[
\{\beta_n  z^n\}_{n \geq 0},
\]
is an orthonormal basis in $H^2(\beta)$, it follows that
\[
\{\beta_n \theta z^n\}_{n \geq 0},
\]
forms an orthonormal basis in $\mathcal{H}$. Also observe that the
multiplication operator $M_z$ on $\mathcal{H}$ is a bounded
operator. Moreover, it follows that $M_z$ on $\clh$ and $M_{z,k}$ on
$H^2(\beta)$ are unitarily equivalent and hence so are $M_z^*$ and
$M_{z,k}^*$. Since the condition
\[
\liminf_n \beta_n = 0,
\]
is necessary and sufficient for $M_{z,k}^*$ to be hypercyclic on
$H^2(\beta)$ \cite{Salas}, under this assumption we conclude that
$M_z^*$ is also hypercyclic on $\clh$. However, the sufficient
condition for hypercyclicity in Theorem \ref{Main-sub} (and hence,
that of Theorem \ref{Main}) is not satisfied for $M_z^*$ on $\clh$.
Indeed, observe that
\[
k_{\theta}(z, w) = (1+z+z^2+\cdots)\left(\beta_0^2+\beta_1^2z
\bar{w} + \beta_2^2 z^2 \bar{w}^2 + \cdots\right) (1+\bar{w} +
\bar{w}^2 + \cdots),
\]
for all $z, w \in \mathbb{D}$. Then $a_{nn}$, the coefficient of
$z^n\bar{w}^n$ in the above expansion, is given by
\[
a_{n n}=\beta_0^2 + \beta_1^2 +\cdots + \beta_n^2,
\]
for all $n \geq 0$, and so
\[
\liminf_n a_{n n} \neq 0.
\]
On the other hand
\[
a_{n n} = \frac {1} {(n!)^2}\frac{\partial^{2n} k_{\theta}}{\partial
z^n\partial \bar{w}^n}(0,0)
\]
for all $n \geq 0$. This shows that the sufficient condition for
hypercyclicity in Theorem \ref{Main} is not necessary.

The construction above also allows us to work with non-zero
polynomials instead of $\theta$. However, in this case, the
sufficient condition in Theorem \ref{Main-sub} for $M_z^*$ to be
hypercyclic is also a necessary condition (see Section
\ref{sect-CR}).

\section{Necessary conditions and Concluding remarks}\label{sect-CR}

In the setting of analytic reproducing kernel Hilbert spaces many
fundamental and basic questions about dynamics remain unanswered.
For instance, the converse of Theorem \ref{Main} is false in general
(see the example in Subsection \ref{example-1}). However, we can
prove the converse for analytic tridiagonal reproducing kernel
Hilbert spaces.

Choose two sequences of non-zero complex numbers $\mu=\{\mu_n\}$ and
$\nu=\{\nu_n\}$ such that $\{e_n\}_{n \geq 0}$ is an orthonormal
basis of a reproducing kernel Hilbert space (of analytic functions
on $\mathbb{D}$) $\mathcal{H}_{{\mu},{\nu}}$ (see Theorem 1 in
\cite{Adams-McGuire}), where
\[
e_n(z)=\mu_n z^n+\nu_nz^{n+1} \quad \quad \quad  ~~(z \in \D, n\geq
0).
\]
Since $\displaystyle k_{\mu,\nu}(z,w)=\sum_{n\geq 0}
e_n(z)\overline{e_n(w)}$, we get the \textit{tridiagonal kernel} as
(cf. \cite{Adams-McGuire})
\[
 k_{\mu,\nu}(z,w) =
|\mu_0|^2+\sum_{n\geq 1}
(|\mu_n|^2+|\nu_{n-1}|^2)z^n\bar{w}^n+\sum_{n\geq 0}\mu_n\bar{\nu_n}
z^{n}\bar{w}^{n+1}+\sum_{n\geq 0}\bar{\mu_n} \nu_n
z^{n+1}\bar{w}^{n},
\]
for all $z, w \in \D$. The analytic function Hilbert space
$\mathcal{H}_{{\mu},{\nu}}$ is called \textit{tridiagonal space}
(see Adams and McGuire \cite{Adams-McGuire} and Adams, McGuire and
Paulsen \cite{VAM}).

\noindent Further, suppose that
\begin{equation}\label{weights}
\sup_{n\geq 0} |\mu_n/\mu_{n+1}|<\infty ~~,~~\sup_{n\geq 0} |\nu_n/
\mu_{n+1}|<1.
\end{equation}
Then Theorem 5 in \cite{Adams-McGuire} ensures that $M_z$ is bounded
on $\mathcal{H}_{{\mu},{\nu}}$. Note also that, as pointed out in
Theorem 4 in \cite{Adams-McGuire}, forward weighted shifts are not
unitarily equivalent to $M_z$ on $\mathcal{H}_{{\mu},{\nu}}$, in general.

Now let $M_z^*$ be hypercyclic, and let $f: = k( \cdot,0)$. Notice
that
\[
f(z) = a+ b z,
\]
where $a =|\mu_0|^2$ and $b = \mu_0\overline{\nu_0}$. Fix $n \geq 0$
and set
\[
z^n f = \sum_{j\geq 0} \alpha_{j} e_j,
\]
that is
\[
az^n + b z^{n+1} = \sum_{j\geq 0} \alpha_{j} (\mu_j z^j+\nu_j
z^{j+1}),
\]
for some $\alpha_j \in \mathbb{C}$. Comparing coefficients of like
powers, we have
\[
\alpha_{j}= 0, \quad  \alpha_{n}=\frac {a} {\mu_n}, \quad
\alpha_{n+1}=\frac {b} {\mu_{n+1}}-\frac{a} {\mu_n}\frac
{\nu_n}{\mu_{n+1}},
\]
for all $j<n$, and
\[
\alpha_{n+1+k} = - \frac{\nu_{n+k}}{\mu_{n+1 + k}} \alpha_{n+ k},
\]
for all $k \geq 1$. Set $M = \max \{a,|b|\}$. Since $\displaystyle R
:= \sup_m \frac{|\nu_m|}{|\mu_{m+1}|} < 1$, it follows that
\[
|\alpha_{n+1}|^2 \leq \big(\frac {1} {|\mu_{n+1}|}+\frac
{1}{|\mu_n|} \big)^2 M^2,
\]
and so
\[
|\alpha_{n+1+k}|^2 \leq \big(\frac {1} {|\mu_{n+1}|}+\frac
{1}{|\mu_n|} \big)^2 M^2 R^{2k},
\]
for all $k \geq 1$. Since
\[
\|z^n f\|^2_{\mathcal{H}_{{\mu},{\nu}}} = \sum_{j\geq 0}
|\alpha_{j}|^2,
\]
we have
\begin{align*}
\|z^n f\|^2_{\mathcal{H}_{{\mu},{\nu}}} & \leq M^2 \Big (\frac {1}
{|\mu_n|^2} + \big(\frac {1} {|\mu_{n+1}|} + \frac {1}{|\mu_n|}
\big)^2 + \big(\frac {1}{|\mu_{n+1}|} + \frac
{1}{|\mu_n|}\big)^2 R^2
\\
& \phantom{=} + \big(\frac {1}{|\mu_{n+1}|} + \frac {1}{|\mu_n|}
\big)^2 R^4 + \cdots \Big)\\&\leq M^2 \Big (\frac {1} {|\mu_n|^2} +
\frac {(r+1)^2} {|\mu_n|^2}\sum_{k\geq 1} {R}^{2k}\Big)\\&\leq C
\frac {1} {|\mu_n|^2},
\end{align*}
where $\displaystyle r=\sup_n (|\mu_n|/|\mu_{n+1}|)$ and $\displaystyle C=M^2\big(1+(r+1)^2 \sum_k R^{2k}\big)$. Now, if possible, let
\[
\liminf_{n} |\mu_n| > 0.
\]
Then the sequence $\{\|z^nf\|_{{\mathcal{H}_{{\mu},{\nu}}}}\}_{n\geq
0}$ is bounded. Now let $g$ be a non-zero vector in
$\mathcal{H}_{{\mu},{\nu}}$. Since
\[
\langle M_z^{*n}g, f  \rangle_{\mathcal{H}_{{\mu},{\nu}}} = \langle
g, z^nf  \rangle_{\mathcal{H}_{{\mu},{\nu}}},
\]
we deduce that $\{M_z^{*n}g\}_{n\geq 0}$ is a bounded set. This
contradicts the fact that $M_z^*$ has a hypercyclic vector, and
hence
\begin{equation}\label{mu}
\displaystyle \liminf_{n} |\mu_n| = 0.
\end{equation}
Finally, since
\[
\frac {1} {(n!)^2} \frac {\partial^{2n}k_{\mu,\nu}} {\partial z^n
\partial \bar{w}^n}(0,0) = |\mu_n|^2+|\nu_{n-1}|^2 = |\mu_n|^2(1+|\nu_{n-1}^2/\mu_n^2|) \leq 2
|\mu_n|^2,
\]
for all $n \geq 1$, it follows that
\[
\liminf_n \frac {1} {(n!)^2} \frac {\partial^{2n}k_{\mu,\nu}}
{\partial z^n
\partial \bar{w}^n}(0,0) = 0,
\]
which is equivalent to \eqref{mu}.

A similar method applies also to the topological mixing property of $M_z^*$. Here we use the following fact: If $T$ is a mixing operator on a Banach space $X$, then $\displaystyle \lim_n \|T^{*n}x^*\| = \infty$ for all non-zero $x^*\in X^*$, where $X^*$ is the dual of $X$ (see Lemma 2.2(2), \cite{Bonet}). This together with the sufficient conditions in Theorem \ref{Main-sub} yields the following complete characterization:

\begin{theorem}
Consider the space $\mathcal{H}_{{\mu},{\nu}}$ defined as above, and assume that $\{\mu_n\}$ and $\{\nu_n\}$ satisfies the condition \eqref{weights}. Then

(i) $M_z^*$ on $\mathcal{H}_{{\mu},{\nu}}$ is hypercyclic if and only if
\[
\liminf_n \frac {1} {(n!)^2} \frac {\partial^{2n}k_{\mu,\nu}}
{\partial z^n
\partial \bar{w}^n}(0,0) = 0,
\]

(ii) $M_z^*$ on $\mathcal{H}_{{\mu},{\nu}}$ is topologically mixing if
and only if
\[
\lim_n \frac {1} {(n!)^2} \frac {\partial^{2n}k_{\mu,\nu}} {\partial
z^n
\partial \bar{w}^n}(0,0) = 0.
\]
\end{theorem}

We now confine our attention to the (counter-)examples
in Subsection \ref{example-1}. Here we again obtain a complete characterization of
dynamics by replacing the analytic function $\theta$ (defined by
$\theta(z) = \frac{1}{1-z}$) by (operator-valued) analytic
polynomials.

Let $\cle$ be a
Hilbert space and let $P(z)$ be a $\mathcal{B}(\cle)$-valued
(analytic) polynomial. Suppose that
\begin{equation*}
P(z)=A_0+A_1z+\cdots+A_dz^d,
\end{equation*}
for some $A_j \in \clb(\cle)$, $j = 0, \ldots, d$. Now we consider a
scalar-valued analytic (and diagonal) kernel $k$ on $\D$:
\[
k(z, w) = \sum_{n\geq 0}\beta_n^2 z^n \bar{w}^n,
\]
where $\beta_n$ are non-negative numbers, and let $\ke$ denote the
reproducing kernel Hilbert space corresponding to the kernel
\[
K(z, w) = k(z, w) I_{\cle},
\]
for all $z, w \in \D$. Let
\begin{equation}\label{eq-KP}
K_P(z,w)=P(z) K(z, w) {P(w)}^*,
\end{equation}
for all $z, w \in \D$. Then $K_P$ is a $\clb(\cle)$-valued kernel on
$\D$. Set
\[
\clh_P = \{P f : f \in \clh_{\cle}(K)\},
\]
where
\[
(Pf)(z) = P(z) f(z),
\]
for all $f \in \ke$ and $z \in \D$. Clearly $\clh_P$ is a vector
space of $\cle$-valued analytic functions on $\D$. Now suppose that
$A_0$ is injective. If $\displaystyle f=\sum_{n\geq 0} \eta_n z^n\in
\ke$ and
\begin{equation}\label{eq-P}
P(z) f(z) = 0,
\end{equation}
for all $z \in \D$, then
\[
P(0)f(0) = A_0 \eta_0 = 0,
\]
implies that $\eta_0=0$. Now differentiating \eqref{eq-P}, one gets
\[
P^\prime(z)f(z) + P(z)f^\prime (z) = 0,
\]
and hence
\[
0 = P^\prime(0) f(0) + P(0) f^\prime (0) = A_0 \eta_1.
\]
We have $\eta_1 = 0$. Continuing in this way we obtain that $f
\equiv 0$. It now follows that $\clh$ is a Hilbert space with inner
product defined by
\[
\langle Pf, Pg \rangle_{\mathcal{H}_P}: = \langle f,g\rangle_{\ke},
\]
for all $f$ and $g$ in $\ke$.

Now we show that $\clh_P$ is the reproducing kernel Hilbert space
corresponding to the kernel $K_P$. Let $\{P f_n\} \subseteq \clh_P$,
and let $P f_n \raro 0$ in $\clh_P$ as $n \raro \infty$. Since
\[
\|P f_n\|_{\clh_P} = \|f_n\|_{\ke},
\]
for all $n \geq 0$, it follows that $f_n \raro 0$ as $n \raro
\infty$ in $\ke$. Now, since $\ke$ is a reproducing kernel Hilbert
space, the evaluation operators are continuous. Therefore
\[
f_n(w) \raro 0 \quad \quad (w \in \D),
\]
in $\cle$ as $n \raro \infty$. Hence
\[
P(w) f_n(w) \raro 0 \quad \quad (w \in \D),
\]
in $\cle$ as $n \raro \infty$. This implies that the evaluation
operators on $\clh_P$ are bounded. Hence $\clh_P$ is a reproducing
kernel Hilbert space. Now, for any $\eta \in \cle$ and $w \in \D$ we
have that
\[
K_P(\cdot, w) \eta = P (k(\cdot, w) P(w)^*\eta) \in \mathcal{H}_P,
\]
and
\[
\begin{split}
\langle Pf, P (k(\cdot, w) P(w)^*\eta)  \rangle_{{\clh}_P} & =
\langle f, k(\cdot, w)P(w)^*\eta \rangle_{\clh_{\cle}(K)}
\\
& = \langle f(w),P(w)^*\eta \rangle_\cle\\&=\langle
P(w)f(w),\eta\rangle_{\cle},
\end{split}
\]
for all $f \in \ke$. Hence $K_P$ is the kernel of $\clh_P$.

Now, we observe that if $M_{z, K}$ (the multiplication operator on
$\ke$ by the coordinate function $z$) is bounded, then $M_{z, K}$ on
$\ke$ and $M_z$ on $\clh_P$ are unitarily equivalent. Moreover,
since $M_{z, K}$ on $\ke$ and $M_{z, k} \otimes I$ on $\clh(k)
\otimes \cle$ are unitarily equivalent, by a result of
Mart\'{i}nez-Gim\'{e}nez and Peris (Proposition 1.14 in
\cite{Gimenez-Peris}) and Salas \cite{Salas}, it follows that
$M_z^*$ is hypercyclic on $\mathcal{H}_P$ if and only if
\[
\liminf_n \beta_n=0.
\]
Now equating the coefficients of $z^n\bar{w}^n$, say $C_{nn}$, in
the expansion of $K_P(z,w)$  in \eqref{eq-KP}, we have (as in
Subsection \ref{example-1})
\[
C_{nn}=A_0A_0^*\beta_n^2+A_1A_1^*\beta_{n-1}^2+\cdots+A_dA_d^*\beta_{n-d}^2,
\]
for all $n\geq d$. Since $M_{z, k}$ is bounded, it follows that
\[
\sup_k \frac {\beta_k} {\beta_{k+1}}<\infty,
\]
and hence, there exists $M > 0$ such that
\[
\|C_{nn}\|_{\clb(\cle)} \leq  M~\beta_{n}^2,
\]
for all $n\geq 1$. Thus, if $M_z^*$ is hypercyclic on
$\mathcal{H}_P$, then, by the above observation, we see that
\[
\liminf_n\|C_{nn}\|_{\clb(\cle)} = 0.
\]
Hence, by
\[
C_{nn}= \frac {1} {(n!)^2} \frac {\partial^{2n}K_P} {\partial z^n
\partial \bar{w}^n}(0,0),
\]
for all $n \geq 0$, it follows that
\[
\liminf_{n}\left(\frac {1} {(n!)^2}\Big< \Big( \frac{\partial^{2n}
K_P}{\partial z^n\partial \bar{w}^n} (0,0) \Big) \eta,\eta
\Big>_{\cle}\right)=0,
\]
uniformly in $\eta\in \cle_1$, where $\cle_1$ is the open unit ball
in $\cle$. This proves the converse of the hypercyclicity part in
Theorem \ref{Main}. We have therefore shown the following result:

\begin{theorem}\label{thm-NS}
Let $\cle$ be a Hilbert space, $k$ on $\D$ be a scalar-valued kernel
and
\[
k(z, w) = \sum_{n\geq 0}\beta_n^2 z^n \bar{w}^n,
\]
where $\beta_n$ are non-negative numbers, and
\[
\displaystyle \limsup_{n} \frac{\beta_{n+1}}{\beta_n} \leq 1.
\]
Let $\ke$ denote the reproducing kernel Hilbert space corresponding
to the kernel
\[
K(z, w) = k(z, w) I_{\cle},
\]
for all $z, w \in \D$. Suppose that $\{A_j\}_{j=0}^d \subseteq
\clb(\cle)$, $A_0$ is injective,
\begin{equation*}
P(z)=A_0+A_1z+\cdots+A_dz^d,
\end{equation*}
and let $\clh_P$ denote the reproducing kernel Hilbert space
corresponding to the kernel function
\[
K_P(z,w)=P(z) K(z, w) {P(w)}^* \quad \quad (z, w \in \D).
\]
If $M_z$ is bounded on $\mathcal{H}_P$, then the following are
equivalent:

\noindent (i) $M_z^*$ is hypercyclic.\\
(ii) $\displaystyle \liminf_{n}\left(\frac {1} {(n!)^2}\Big< \Big(
\frac{\partial^{2n} K_P}{\partial z^n\partial \bar{w}^n} (0,0) \Big)
\eta,\eta \Big>_{\cle}\right)=0$ uniformly in $\eta \in \cle_1$, where $\cle_1$ is the open unit ball in $\cle$.\\
(iii) $\displaystyle \liminf_n \beta_n=0$.
\end{theorem}

The same method of proof also applies to the scalar-valued reproducing
kernel Hilbert space $\clh_P$ corresponding to the kernel function
\[
k_P(z,w)=P(z)k(z, w) \overline{P(w)} \quad \quad (z, w \in \D),
\]
where
\[
P(z)=a_0 +a_1z + \cdots + a_dz^d,
\]
is a non-zero scalar polynomial and $k$ is defined as in Theorem
\ref{thm-NS} (observe that $k_P(z,w)$ is not a diagonal kernel in
general). Here $\clh_P = \{P f : f \in \clh(k)\}$ and the inner
product is given by $\langle Pf, Pg \rangle_{\clh_P} = \langle f, g
\rangle_{\clh(k)}$ for all $f, g \in \clh(k)$.

\begin{theorem}
If $M_z$ is bounded on $\mathcal{H}_P$, then the following are
equivalent:

\noindent (i) $M_z^*$ is hypercyclic.\\
(ii) $\displaystyle \liminf_n \frac {1} {(n!)^2} \Big(\frac
{\partial^{2n}k} {\partial z^n
\partial \bar{w}^n}(0,0)\Big)=0$.\\
(iii) $\displaystyle \liminf_n \beta_n=0$.
\end{theorem}

It would be therefore very interesting to determine the class of
vector valued analytic kernel functions for which the converse of
Theorem \ref{Main} holds.

\vspace{0.3in}

{\bf Acknowledgments.} The authors would like to express their
sincerest gratitude to the referee for careful reading of the
manuscript, thoughtful remarks, and for many valuable suggestions.
The authors are also grateful to Deepak Pradhan for a careful
reading of the manuscript and valuable comments. The first author is
grateful to the Indian Statistical Institute Bangalore for offering
a visiting position. The first author is also supported in part by
the NBHM post doctoral fellowship, File No.
0104/43/2017/RD-II/15354. The second author is supported in part by
the Mathematical Research Impact Centric Support (MATRICS) grant,
File No : MTR/2017/000522, by the Science and Engineering Research
Board (SERB), Department of Science \& Technology (DST), Government
of India, and NBHM (National Board of Higher Mathematics, DAE, India)
Research Grant NBHM/R.P.64/2014.

\bibliographystyle{amsplain}

\end{document}